\documentclass[12pt]{article}
\usepackage{graphicx}
\usepackage{amsmath,amsthm,amssymb,enumerate}
\usepackage{euscript,mathrsfs}
\usepackage{color}
\usepackage[left=2cm,right=2cm,top=3.5cm,bottom=3.5cm]{geometry}
\usepackage{color}
\usepackage{esint}
\usepackage[framemethod=tikz]{mdframed}
\allowdisplaybreaks

\usepackage{soul}

\catcode`\@=11 \@addtoreset{equation}{section}

\catcode`\@=12

\newtheorem{Theorem}{Theorem}[section]
\newtheorem{Proposition}[Theorem]{Proposition}
\newtheorem{Lemma}[Theorem]{Lemma}
\newtheorem{Corollary}[Theorem]{Corollary}

\theoremstyle{definition}
\newtheorem{Definition}[Theorem]{Definition}

\newtheorem{Remark}[Theorem]{Remark}

\newcommand{\bTheorem}[1]{
	\begin{Theorem} \label{T#1} }
	\newcommand{\eT}{\end{Theorem}}

\newcommand{\bProposition}[1]{
	\begin{Proposition} \label{P#1}}
	\newcommand{\eP}{\end{Proposition}}

\newcommand{\bLemma}[1]{
	\begin{Lemma} \label{L#1} }
	\newcommand{\eL}{\end{Lemma}}

\newcommand{\bCorollary}[1]{
	\begin{Corollary} \label{C#1} }
	\newcommand{\eC}{\end{Corollary}}

\newcommand{\bRemark}[1]{
	\begin{Remark} \label{R#1} }
	\newcommand{\eR}{\end{Remark}}

\newcommand{\bDefinition}[1]{
	\begin{Definition} \label{D#1} }
	\newcommand{\eD}{\end{Definition}}

\newcommand{\Del}{\Delta_x}

\newcommand{\wmT}{\widetilde{\mathcal{T}}}

\newcommand{\data}{{\rm data}}

\newcommand{\bfphi}{\boldsymbol{\varphi}}

\newcommand{\bFormula}[1]{
	\begin{equation} \label{#1}}
	\newcommand{\eF}{\end{equation}}

\newcommand{\avintO}[1]{\fint_{\Omega} #1 \dx}
\newcommand{\vun}{\vu_n}

\newcommand{\Ov}[1]{\overline{#1}}

\newcommand{\aleq}{\stackrel{<}{\sim}}

\newcommand{\vr}{\varrho}

\newcommand{\tvu}{{\tilde \vu}}

\newcommand{\vt}{\vartheta}
\newcommand{\vu}{\vc{u}}

\newcommand{\vc}[1]{{\bf #1}}

\newcommand{\Div}{{\rm div}_x}
\newcommand{\Grad}{\nabla_x}

\newcommand{\dx}{\,{\rm d} {x}}

\newcommand{\dt}{\,{\rm d} t }

\newcommand{\intO}[1]{\int_{\Omega} #1 \ \dx}

\newcommand{\D}{{\rm d}}

\newcommand{\vtB}{\vt_B}

\newcommand{\br}{ \nonumber \\ }

\def\softd{{\leavevmode\setbox1=\hbox{d}%
		\hbox to 1.05\wd1{d\kern-0.4ex{\char039}\hss}}}
\definecolor{Cgrey}{rgb}{0.85,0.85,0.85}
\definecolor{Cblue}{rgb}{0.50,0.85,0.85}
\definecolor{Cred}{rgb}{1,0,0}
\definecolor{fancy}{rgb}{0.10,0.85,0.10}

\newcommand\Cbox[2]{%
	\newbox\contentbox%
	\newbox\bkgdbox%
	\setbox\contentbox\hbox to \hsize{%
		\vtop{
			\kern\columnsep
			\hbox to \hsize{%
				\kern\columnsep%
				\advance\hsize by -2\columnsep%
				\setlength{\textwidth}{\hsize}%
				\vbox{
					\parskip=\baselineskip
					\parindent=0bp
					#2
				}%
				\kern\columnsep%
			}%
			\kern\columnsep%
		}%
	}%
	\setbox\bkgdbox\vbox{
		\color{#1}
		\hrule width  \wd\contentbox %
		height \ht\contentbox %
		depth  \dp\contentbox
		\color{black}
	}%
	\wd\bkgdbox=0bp%
	\vbox{\hbox to \hsize{\box\bkgdbox\box\contentbox}}%
	\vskip\baselineskip%
}

\mdfdefinestyle{MyFrame}{%
	linecolor=black,
	outerlinewidth=1pt,
	roundcorner=5pt,
	innertopmargin=\baselineskip,
	innerbottommargin=\baselineskip,
	innerrightmargin=10pt,
	innerleftmargin=10pt,
	backgroundcolor=white!20!white}

\newenvironment{giuliorev}{\color{magenta}}{\color{black}}
\newcommand{\III}{\begin{giuliorev}}
\newcommand{\EEE}{\end{giuliorev}}
 
\definecolor{rosso}{rgb}{0.85,0,0}

\begin{document}


\title{The Oberbeck--Boussinesq approximation and Rayleigh--B\' enard convection
revisited}

\author{Eduard Feireisl$^1$
		\thanks{The work of E.F. was partially supported by the
			Czech Sciences Foundation (GA\v CR), Grant Agreement
			24-11034S. The Institute of Mathematics of the Academy of Sciences of
			the Czech Republic is supported by RVO:67985840. } 
		 \and Elisabetta Rocca$^2$ \and Giulio Schimperna$^2$ \thanks{The work of E.R. and G.S. has been partially supported by the MIUR-PRIN Grant 2020F3NCPX
``Mathematics for industry 4.0 (Math4I4)''. The present paper also benefits from the support of
the GNAMPA (Gruppo Nazionale per l'Analisi Matematica, la Probabilit\`a e le loro Applicazioni)
of INdAM (Istituto Nazionale di Alta Matematica).}}

\date{}

\maketitle

\medskip

\centerline{$^1$Institute of Mathematics of the Academy of Sciences of the Czech Republic}

\centerline{\v Zitn\' a 25, CZ-115 67 Praha 1, Czech Republic}

\medskip

\centerline{$^2$Dipartimento di Matematica, Universit\`a degli studi di Pavia and IMATI-C.N.R.}

\centerline{via Ferrata, 5-27100 Pavia, Italy}

\begin{abstract}
	
We consider the Oberbeck--Boussinesq approximation driven by an inhomogeneous 
temperature distribution on the boundary of a bounded fluid domain. The relevant boundary conditions are perturbed by a non--local term arising in the incompressible limit of the Navier--Stokes--Fourier system. The long time behaviour of the resulting initial/boundary value problem is investigated.

\end{abstract}


{\bf Keywords:} Oberbeck--Boussinesq approximation, Rayleigh-B\' enard convection, long--time behaviour, non--local boundary conditions


\section{Introduction}
\label{i}

The \emph{Oberbeck--Boussinesq (OB) system} is frequently used 
as a suitable approximation of the motion of an incompressible fluid 
driven by thermal convection, see e.g. the surveys of 
Bodenschatz, Pesch, and Ahlers \cite{BoPeAh}, 
Ecke and Shishkina \cite{EckShi}. The system of equations can be rigorously derived as a singular low Mach/moderate stratification limit of the 
primitive Navier--Stokes--Fourier system describing the motion of a compressible 
heat conducting fluid, cf. \cite{BelFeiOsch}, \cite{FanFei2023}, \cite[Chapter 5]{FeNo6A}. In the Rayleigh--B\' enard convection problem, the fluid motion is driven by the mutual interaction of the gravitational force and the thermal gradient enforced by the Dirichlet boundary conditions for the temperature. 
The singular limit process then produces a rather unexpected non--local 
boundary term so far completely ignored in the literature, see \cite{BelFeiOsch} for the well--prepared and \cite{FanFei2023} for the ill--prepared data. 
Accordingly, the resulting limit problem consists of the Navier--Stokes system 
for the fluid velocity $\vu = \vu(t,x)$, 
\begin{align}
	\Div \vu = 0, \label{i2} \\
	\partial_t \vu + \Div (\vu \otimes \vu) + \Grad \Pi &= \mu \Del \vu - \Theta \Grad G,\ \mu > 0, \label{i1}
\end{align} 
coupled with the equation for the temperature deviation $\Theta = \Theta(t,x)$,
\begin{equation} \label{i3} 
\partial_t \Theta + \Div (\Theta \vu) + a \Div (G \vu) = \kappa \Del \Theta,\ 
a \in R,\ \kappa > 0 
\end{equation}
for $t > 0$, $x \in \Omega$, bounded and sufficiently regular subset of $R^d$, $d=2,3$. The problem is supplemented 
with the \emph{Rayleigh--B\' enard (RB) type boundary conditions}:
\begin{align} 
	\vu|_{\partial \Omega} = 0,  \label{i4} \\
	\Theta|_{\partial \Omega} = \vtB  - \alpha  \avintO{ \Theta },\ \alpha > 0,
	  \label{i5}
	\end{align}
where $	\fint \equiv \frac{1}{|\Omega|} \int_\Omega$. The function 
$\vtB = \vtB(x)$ is the boundary distribution of the original temperature deviation, while the integral term is an inevitable product of the singular limit process. The value 
of the parameter $\alpha$ depends on the reference state of the system and will be discussed later. As both $\Theta$ and $\vtB$ represent the deviations from a 
background temperature, they need not be non--negative.

The function $G = G(x)$ represents a gravitational potential. In addition, we suppose
\begin{equation} \label{i6}
	\Del G = 0 \ \mbox{in}\ \Omega,\ \intO{ G } = 0.
\end{equation}	
Replacing $\Theta \approx \Theta + aG$, $\vtB \approx \vtB + aG$, we may rewrite the above problem in the form:
\begin{align}
	\partial_t \vu + \Div (\vu \otimes \vu) + \Grad \Pi &= \mu \Del \vu - \Theta \Grad G, \label{i7}\\
	\Div \vu &= 0, \label{i8}\\
	\partial_t \Theta + \Div (\Theta \vu)&= \kappa \Del \Theta, \label{i9}
\end{align}
with the boundary conditions 
\begin{align} 
	\vu|_{\partial \Omega} &= 0, \label{i10} \\
	\Theta|_{\partial \Omega} &= \vtB  - \alpha  \avintO{ \Theta }. \label{i11}
\end{align}

The OB system \eqref{i7}--\eqref{i9}, with the standard Dirichlet boundary 
conditions \eqref{i10} and \eqref{i11} with $\alpha = 0$, represents an iconic 
model of turbulent fluid motion examined in a large number of studies. In particular, Foias, Manley, and Temam \cite{FoMaTe} established the existence of 
global attractors and estimates on their fractal dimension. The problem has been revisited by Cabral, Rosa, and Temam \cite{CaRoTe}, where more complex geometries of the fluid domain as well as general driving forces were included. Wang \cite{WangX} considered the same problem in the large Prandtl number regime, where 
the fluid velocity becomes regular for long times. Similar ideas were exploited 
by Birnir and Svanstedt \cite{BirSva}. Notice that here $\nabla G$ could also include a centrifugal force (cf., e.g., \cite{EckShi}). Such a model, however, cannot be obtained by the limit passage performed in \cite{BelFeiOsch}. 

The \emph{existence} theory for the problem \eqref{i7}--\eqref{i11} was developed in \cite{AbbFei22}. The result inherits the well known difficulties related to 
the Navier--Stokes system. Regular solutions exist on a maximal time interval 
$[0, T_{\rm max})$, where $T_{\rm max} = \infty$ if $d=2$. The problem admits global in time weak solutions if $d =2,3$. Our goal is to study the long--time 
behaviour of the weak solutions for $t \to \infty$. In particular, we address the following issues:
\begin{itemize}
	\item Dissipativity and the existence of a bounded absorbing sets.
	\item Asymptotic compactness of global--in--time trajectories.
	\item Existence of stationary statistical solutions and convergence of ergodic averages.
	\item Convergence to equilibria.
\end{itemize}

In comparison with the existing literature, the principal difficulty is the presence of the integral mean in the boundary condition \eqref{i11}. In contrast 
with \cite{FoMaTe}, \cite{WangX}, the heat equation \eqref{i9} accompanied 
with the non--local boundary condition \eqref{i11} does not admit any obvious 
maximum/minimum principle. Establishing uniform bounds on $\Theta$ is therefore the heart of the present paper.

The paper is organized as follows. In Section \ref{w}, we collect the preliminary material concerning the existence and basic properties of the weak solutions to the OB system. Section \ref{U} contains the proof of uniform bounds on the temperature -- a variant of the maximum/minimum principle for parabolic equations with non--local boundary conditions that may be of independent interest. 
In Section \ref{B}, we show Levinson dissipativity or the existence of a bounded absorbing set independent of the size of specific initial data, while, in Section \ref{B2} we prove the existence of statistical solutions. Section \ref{C} addresses the convergence of global in time solutions to equilibria. Of course, the result is conditioned by smallness of the associated equilibrium solution and not expected to hold for general (large) data. In particular, we recover a ``Boussinesq analogy'' of the convergence result of Abbatiello, Bul\' \i \v cek, Kaplick\' y \cite{AbbBulKap} for small driving forces.  
   
\section{Preliminary material, weak formulation}
\label{w}

We assume the following general hypotheses on the data (cf.~\cite{AbbFei22})
\begin{align}
\label{d1}
&\vu_0 \in  L^2(\Omega; R^d), \quad \Div \vu_0 = 0,\\
\label{d2}
 &\Theta_0\in W^{2,q}(\Omega)\cap L^\infty(\Omega), \quad \vtB\in W^{2,q}(\partial\Omega),\quad (\Theta_0)_{|\partial\Omega}=\vtB-\alpha\avintO{ \Theta_0 },\\
 \label{d3}
 &G\in W^{1,\infty}(\Omega), \quad 	\Del G = 0 \ \mbox{in}\ \Omega,\ \intO{ G } = 0,
\end{align}
for a certain $q>1$. We adopt the standard weak formulation for the Navier--Stokes system. The velocity component belongs to the Ladyzhenskaya class 
\begin{equation} \label{w1}
\vu \in L^\infty_{\rm loc}([0,\infty); L^2(\Omega; R^d)) \cap 
L^2_{\rm loc}([0, \infty); W^{1,2}_0 (\Omega; R^d)), 
\end{equation}
\begin{align} 
	\Div \vu &= 0 \ \mbox{a.a. in}\ (0, \infty) \times \Omega, \label{w2}\\ 
\left[ \intO{ \vu \cdot \bfphi } \right]_{t = 0}^{t = \tau} &=
\int_0^\tau \intO{ \Big( \vu \cdot \partial_t \bfphi + (\vu \otimes \vu) : \Grad \bfphi - 
	\mu \Grad \vu \cdot \Grad \bfphi - \Theta \Grad G \cdot \bfphi \Big) } \dt 
\label{w3}
\end{align}
for any $\tau > 0$, and any $\bfphi \in C^1_c([0, \infty) \times \Omega; R^d)$, $\Div \bfphi = 0$.

The temperature $\Theta$ belongs to the class 
\begin{equation} \label{w4}
	\Theta \in L^2_{\rm loc}([0,T); W^{1,2}(\Omega)) \cap 
	L^\infty_{\rm loc}([0, \infty) \times \Ov{\Omega}), 
	\end{equation}
and 
\begin{align} 
\Theta + \alpha \avintO{\Theta} - \widetilde{\vt}_B
&\in L^2_{\rm loc}([0, \infty); W^{1,2}_0 (\Omega)), \label{w5} \\ 	
\left[ \intO{ \Theta \varphi } \right]_{t = 0}^{t = \tau} 
&= \int_0^\tau \intO{ \Big( \Theta \partial_t \varphi + \vu \Theta \cdot \Grad \varphi - \kappa \Grad \Theta \cdot \Grad \varphi \Big) } \dt 
\label{w6}		 
\end{align}
for any $\tau > 0$, and any $\varphi \in C^1_{c}([0, \infty) \times \Omega)$,
where $\widetilde{\vt}_B$ is the unique harmonic extension of the boundary data, namely
\[
\Del \widetilde{\vt}_B = 0 \ \mbox{in}\ \Omega,\ \widetilde{\vt}_B|_{\partial \Omega} = \vtB.
\]
To simplify notation, we shall use the symbol $\vtB$ for both the boundary datum and
its harmonic extension hereafter. In addition, we suppose $\vtB$ is sufficiently smooth, say 
	of class $C^2(\Ov{\Omega})$. 

It follows from the above that 
\begin{equation} \label{w8}
\vu \in C_{\rm weak}([0, \infty); L^2(\Omega; R^d)), \ 
\Theta \in C_{\rm weak}([0, \infty); L^q(\Omega; R^d)),\ 1 \leq q < \infty.
\end{equation}
In particular, 
\begin{equation} \label{w10}
	t \mapsto \avintO{ \Theta (t, \cdot) } \in C[0, \infty).
\end{equation}

In addition, we impose the relevant energy inequality for the Navier--Stokes system: 
\begin{align} \nonumber
	& \left[ \psi \intO{ \frac{1}{2} |\vu |^2 } \right]_{t = 0}^{t = \tau} 
	+ \int_{0}^\tau \psi \intO{ \mu |\Grad \vu |^2 } \dt 
	-\int_0^\tau \partial_t \psi \intO{ \frac{1}{2} |\vu|^2 } \dt \\
 \label{w7}
	& \leq - \int_{0}^\tau \psi 
	\intO{ \Theta \Grad G \cdot \vu } \dt 
\end{align}	  
for any $\tau > 0$, and any $\psi \in C^1_{c}[0, \infty)$, $\psi \geq 0$.

The existence of global in time weak solutions was proved in \cite[Theorem 2.3]{AbbFei22}.

\subsection{Energy balance for the temperature}
\label{web}

We first introduce a new function 
\begin{equation} \label{w10a}
\mathcal{T} = \Theta + \alpha \avintO{ \Theta } - {\vt}_B.
\end{equation}
Then, in accordance with \eqref{w4}--\eqref{w6}, 
\begin{equation} \label{w11}
\mathcal{T} \in L^2_{\rm loc}([0, \infty); W^{1,2}_0 (\Omega)) \cap 
L^\infty_{\rm loc}([0, \infty) \times \Omega) \cap C_{\rm weak}([0, \infty); 
L^q(\Omega)),\ 1 \leq q < \infty.
\end{equation}
Moreover, the weak formulation \eqref{w6} can be rewritten in the form
\begin{align} 
	\left[ \intO{ \left[ \mathcal{T} - \frac{\alpha}{1 + \alpha} \avintO{ 
			\mathcal{T} } \right] \varphi  } \right]_{t = 0}^{t = \tau} 
	&= \int_0^\tau \intO{ \left[ \mathcal{T} - \frac{\alpha}{1 + \alpha} \avintO{ 
			\mathcal{T} } \right] \partial_t \varphi } \dt \br 
		 & + \int_0^\tau \intO{ \Big( \left[ \mathcal{T} + \vtB \right] \vu \cdot \Grad \varphi 
		  - \kappa \Grad \mathcal{T} \cdot \Grad \varphi \Big) } \dt \label{w12}
\end{align}
for any $\varphi \in C^1_{c}([0, \infty) \times \Omega)$. 
Introducing a strictly positive bounded self--adjoint operator $\Lambda$, 
\[
\Lambda (Z) = Z - \frac{\alpha}{1 + \alpha} \avintO{Z},\ Z \in L^2(\Omega),
\]
we may rewrite
the integral identity \eqref{w12} in a concise form
 \begin{align} 
 	\left[ \intO{ \Lambda( \mathcal{T} ) \varphi  } \right]_{t = 0}^{t = \tau} 
 	= \int_0^\tau \intO{ \Lambda(\mathcal{T}) \partial_t \varphi } \dt 
 	+ \int_0^\tau \intO{ \Big( \left[ \mathcal{T} + \vtB \right] \vu  \cdot \Grad \varphi 
 	- \kappa \Grad \mathcal{T} \cdot \Grad \varphi \Big) } \dt.  \label{w13}
 \end{align}

Our goal is to justify the choice $\varphi  = \mathcal{T}$ as an admissible test function in \eqref{w13}. To this end, observe that the quantity 
\[
F = \Div \Big(\vu \left[ \mathcal{T} + \vtB \right]  - \kappa \Grad \mathcal{T}  
\Big)
\]
can be interpreted as a bounded function in $L^2_{\rm loc}([0, \infty); W^{-1,2}( 
\Omega))$. Moreover, extending 
\[
\mathcal{T}(t, \cdot) = \mathcal{T}(0, \cdot),\ 
F (t, \cdot) = 0 \ \mbox{for}\ t \leq 0 
\]
we obtain 
\begin{equation} \label{w14}
	\frac{\D }{\dt} { \Lambda(\mathcal{T}) } = 
	- F \ \mbox{in}\ L^2_{\rm loc}(R; W^{-1,2}(\Omega))\,,
	\end{equation}
	where we have used the standard Gelfand triple identification 
\[
W^{1,2}_0 \subset L^2 \approx L^2 \subset W^{-1,2}.
\]

Next, we regularize \eqref{w12} in the time variable by means of a convolution 
with a family of regularizing kernels $(\xi_\delta)$, $[h]_\delta \equiv h * \xi_\delta$, obtaining	 
\[
	\frac{\D }{\dt} { \Lambda(\mathcal{T}_\delta) } =
	- \Div \Big[ \vu[ \mathcal{T} + \vtB] - \kappa \Grad \mathcal{T} \Big]_\delta  \ \mbox{in}\ R.  
\]
Testing the above relation by 
$\mathcal{T}_\delta$, we get 
\[
\frac{1}{2} \frac{\D }{\dt} \intO{ \Lambda(\mathcal{T}_\delta) \mathcal{T}_\delta }
= \intO{ \Big[ \vu[ \mathcal{T} + \vtB] - \kappa \Grad \mathcal{T} \Big]_\delta \cdot 
	\Grad \mathcal{T}_\delta }.
\]

Finally, letting $\delta \to 0$, we get the desired conclusion
\begin{align} \label{w15}
\left[ \psi \intO{ \frac{1}{2} \Lambda (\mathcal{T} ) \mathcal{T} } \right]_{t = 0}^{t = \tau} 
&+ \int_0^\tau \psi \intO{ \kappa |\Grad \mathcal{T} |^2 } \dt 
- \int_0^\tau \partial_t \psi \intO{ \frac{1}{2} \Lambda (\mathcal{T} ) \mathcal{T} } \dt \br 
&= \int_0^\tau \psi \intO{ \vu \Big( \mathcal{T} + \vtB \Big) \cdot \Grad \mathcal{T}  } = \int_0^\tau \psi \intO{ \vu \cdot \Grad \mathcal{T} \vtB } \dt,
\end{align} 	
for any $\tau > 0$ and $\psi \in C^1_c[0, \infty)$, where we have used $\Div \vu = 0$. 

\begin{Remark} \label{Rw1}
The above inequality holds for $a.a.$ $\tau > 0$ as $\mathcal{T}$ is only weakly continuous in time. As we shall see below, however, strong continuity of $\mathcal{T}$ and, consequently, of $\Theta$ can be established. 
\end{Remark}	

\subsection{Higher regularity of the temperature}

Repeating the arguments of the previous part  we first establish an auxiliary 
result concerning uniqueness of the temperature $\Theta$ in terms of the velocity 
field $\vu$. 

\begin{Proposition}[{\bf Uniqueness of $\Theta$}] \label{Pw1}
	Let the velocity $\vu$, 
	\[ 
	\vu \in L^\infty_{\rm loc}([0,\infty); L^2(\Omega; R^d)) \cap 
	L^2_{\rm loc}([0, \infty); W^{1,2}_0 (\Omega; R^d)),\ \Div \vu = 0,  
	\]
and the initial distribution of the temperature
\[
\Theta_0 \in L^\infty (\Omega)
\]
be given.

Then the problem \eqref{w5}, \eqref{w6} admits at most one solution 
$\Theta$, $\Theta(0, \cdot) = \Theta_0$,  in the class 
\[
		\Theta \in L^2_{\rm loc}([0,T); W^{1,2}(\Omega)) \cap 
	L^\infty_{\rm loc}([0, \infty) \times \Ov{\Omega}).
\]	
	\end{Proposition}

\begin{proof}
Let $\Theta_1$, $\Theta_2$ be two solutions of \eqref{w5}, \eqref{w6}. Passing to 
$\mathcal{T}_1$, $\mathcal{T}_2$ defined through formula \eqref{w10a}, we may 
repeat step by step the arguments of Section \ref{web} obtaining
\begin{align*}
 & \frac{1}{2} \intO{ \Lambda( \mathcal{T}_1 - \mathcal{T}_2 )( \mathcal{T}_1 - \mathcal{T}_2 )(\tau, \cdot) } + 
\kappa \int_0^\tau \intO{ |\Grad (\mathcal{T}_1 - \mathcal{T}_2) |^2 } \dt\cr
& \qquad = \int_0^\tau \intO{ \vu (\mathcal{T}_1 - \mathcal{T}_2) \cdot \Grad  (\mathcal{T}_1 - \mathcal{T}_2) } \dt
\end{align*}
for a.a. $\tau > 0$, cf. \eqref{w15}. Seeing that the regularity of $\vu$ and $\Theta$ is sufficient to justify the identity
\[
\int_0^\tau \intO{ (\mathcal{T}_1 - \mathcal{T}_2) \vu \cdot 
	\Grad  (\mathcal{T}_1 - \mathcal{T}_2) } \dt = 
\frac{1}{2}	\int_0^\tau \intO{ \vu  \cdot 
		\Grad  (\mathcal{T}_1 - \mathcal{T}_2)^2 } \dt = 0
\] 
we conclude $\mathcal{T}_1 = \mathcal{T}_2$.
	\end{proof}

Proposition \ref{Pw1} allows us to deduce higher regularity of the temperature 
via a bootstrap argument. Specifically, we may regularize the velocity $\vu$ 
to obtain more regular solutions, where pointwise estimates are available. In 
particular, seeing that the boundary values 
\[
 \Theta |_{\partial \Omega} = \vtB - \alpha \avintO{ \Theta} \in 
C_{\rm loc}([0, \infty); W^{2,q}(\partial\Omega)), 
\]
and 
\[
\vu \cdot \Grad \Theta \in L^q_{\rm loc}([0,T) \times \Ov{\Omega}) 
\ \mbox{for a certain}\ q > 1, 
\]	
we may use the maximal $L^p-L^q$ regularity estimates established by Denk, Hieber, and Pr\" uss \cite{DEHIEPR}	to conclude 
\begin{equation} \label{w17}
	\partial_t \Theta , \ D^2_{x,x} \Theta \in L^q_{\rm loc}([0, \infty) \times \Ov{\Omega}) \ \mbox{for a certain}\ q > 1
\end{equation}
as long as the initial datum $\Theta_0$ is regular enough, 
 as specified in \eqref{d2}. 

In particular, the equations \eqref{i9} as well as \eqref{w12}, \eqref{w13} are satisfied in the strong sense a.a. in $(0,T) \times \Omega$. In addition, 
\begin{equation} \label{w18}
\Theta \in C([0, \infty); L^p(\Omega)) \ \mbox{for any}\ 1 \leq p < \infty, 
\end{equation}
which, in particular, justifies validity of \eqref{w15} for any $\tau > 0$.

\subsection{Validity of the maximum principle for $\Theta$}
\label{wmp}

Given a positive time $T > 0$, fix two constants satisfying
\[
\underline{\Theta}_{\partial \Omega}^T \leq \inf_{[0,T] \times \partial \Omega} 
\Theta \leq \sup_{[0,T] \times \partial \Omega} \Theta \leq 
\Ov{\Theta}^T_{\partial \Omega}. 
\]
Similarly to \cite[Section 2.2]{FoMaTe}, Wang \cite{WangX}, we claim that any  solution 
$\Theta$ of \eqref{w4}--\eqref{w6} satisfies the following inequalities:
 \begin{align} 
\left[ \intO{ \frac{1}{2} \left([ \Theta - \Ov{\Theta}^T_{\partial \Omega}]^+ \right)^2 } \right]_{t=0}^{t= \tau} 	+ \int_0^\tau \intO{ | \Grad [ \Theta - \Ov{\Theta}^T_{\partial \Omega}]^+ |^2 } \dt &\leq 0, \br
\left[ \intO{ \frac{1}{2} \left([ \Theta - \underline{\Theta}^T_{\partial \Omega}]^- \right)^2 } \right]_{t=0}^{t= \tau}	+ \int_0^\tau \intO{ | \Grad [ \Theta - \underline{\Theta}^T_{\partial \Omega}]^- |^2 } \dt &\leq 0 
 	\label{w22}
\end{align}
for any $0 \leq \tau \leq T$. Indeed the above relations follow by simple multiplication of the equation \eqref{i9} by 
$[ \Theta - \Ov{\Theta}^T_{\partial \Omega}]^+$, $[ \Theta - \Ov{\Theta}^T_{\partial \Omega}]^-$, respectively. This step can be justified 
at least for weak solutions belonging to the regularity class \eqref{w17}. Extension to less regular initial data follows by density arguments.

Next, we introduce the \emph{parabolic boundary} 
\[
\Gamma_T = \{ 0 \} \times \Ov{\Omega} \cup [0,T] \times \partial \Omega.
\]
As $\Theta \in L^\infty_{\rm loc}([0, \infty) \times \Ov{\Omega})$, 
we may define
\[
\Ov{\Theta}(\tau) = {\rm ess}\sup_{[0, \tau) \times \Omega} \Theta,\ 
\underline {\Theta}(\tau) = {\rm ess}\inf_{[0, \tau) \times \Omega} \Theta. 
\]
Obviously, $\tau \mapsto \Ov{\Theta}[0,\tau]$ is a non--decreasing while 
$\tau \mapsto \underline{\Theta}[0,\tau]$ is a non--increasing function of $\tau$.

If $\vu$, and, consequently, $\Theta$ are smooth, the standard parabolic maximum/minimum principle yields 
\begin{equation} \label{w20}
\Ov{\Theta}(\tau) \leq \sup_{\Gamma_\tau} \Theta,\ 
\underline{\Theta}(\tau) \geq \inf_{\Gamma_\tau} \Theta.
\end{equation} 
Using \eqref{w22} we may extend the maximum/minimum principle to any weak solution $(\vu, \Theta)$ of the OB system. 
 Namely, we have 
	
	\begin{Proposition}[{\bf Maximum/minimum principle}] \label{wP2}
		
		Let $\vu$, $\Theta$ be a weak solution of problem \eqref{w2}, \eqref{w3}, 
		\eqref{w5}, \eqref{w6} belonging to the regularity class \eqref{w1}, \eqref{w4}. 
		
		Then 
		\begin{equation} \label{w21}
			\Ov{\Theta}(\tau) \leq {\rm ess} \sup_{\Gamma_\tau} \Theta,\ 
			\underline{\Theta} (\tau) \geq {\rm ess} \inf_{\Gamma_\tau} \Theta
		\end{equation}
for any $\tau \geq 0$.		
		
		\end{Proposition}

\section{Uniform bounds on the temperature}
\label{U}	

The results obtained in this section are crucial for the long--time behaviour of the OB system and may be also of independent interest. Here and hereafter, we impose a fundamental hypothesis 
\begin{equation} \label{UU}
	0 < \alpha < 1. 
	\end{equation}
A short inspection of the proof in \cite{BelFeiOsch} reveals the value of 
the constant $\alpha$, specifically, 
\[
\alpha = \frac{\lambda (\Ov{\vr}, \Ov{\vt})}{1 - \lambda(\Ov{\vr}, \Ov{\vt})},
\] 
where
\begin{align}
\lambda(\Ov{\vr}, \Ov{\vt}) &= \frac{\Ov{\vt} \beta (\Ov{\vr}, \Ov{\vt} ) }
{\Ov{\vr} c_p(\Ov{\vr}, \Ov{\vt}) } \frac{\partial p(\Ov{\vr}, \Ov{\vt} )}{\partial \vt}, \  
	\beta (\Ov{\vr}, \Ov{\vt} ) = \frac{1}{\Ov{\vr}}  \frac{\partial p(\Ov{\vr}, \Ov{\vt} ) }{\partial \vt} \left( \frac{\partial p(\Ov{\vr}, \Ov{\vt} ) }{\partial \vr} \right)^{-1},\br  
	c_p (\Ov{\vr}, \Ov{\vt} ) &= \frac{\partial e(\Ov{\vr}, \Ov{\vt} ) }{\partial \vt}	+ \Ov{\vr}^{-1} \Ov{\vt}
	\beta(\Ov{\vr}, \Ov{\vt} ) \frac{\partial p(\Ov{\vr}, \Ov{\vt} ) }{\partial \vt}, 
\nonumber	
\end{align}
and where $\beta$ is the coefficient of thermal expansion and $c_p$ the specific heat at constant pressure evaluated for the primitive Navier--Stokes--Fourier system at the reference (constant) density $\Ov{\vr}> 0$ and temperature $\Ov{\vt} > 0$. Thus, by standard manipulations, we get 
\[
\alpha = \frac{c_p}{c_v} -1=\gamma-1<1,
\]
in agreement with \eqref{UU}.

Our first result establishes uniform boundedness of the temperature 
as time approaches infinity in terms of the initial and boundary data. The arguments depend solely on the maximum/minimum 
principle stated in Proposition \ref{wP2}.

\begin{mdframed}[style=MyFrame]
	
	\begin{Theorem}[{\bf Uniform bounds on the temperature}] \label{UT1}
		Let 
		\[
		0 < \alpha < 1.
		\]
Let $\vu$, $\Theta$ be a weak solution of the OB system \eqref{w2}, \eqref{w3}, 
\eqref{w5}, \eqref{w6} belonging to the regularity class \eqref{w1}, \eqref{w4}. 		

Then  
\begin{equation} \label{U1}
| \Theta(t,x) | \leq {\rm ess} \sup_{\Omega} |\Theta(0, \cdot)|  + \frac{2}{1 - \alpha^2} 
\max_{\partial \Omega} |\vtB |
\end{equation} 
for a.a. $(t, x) \in (0, \infty) \times \Omega$.		
		
		\end{Theorem}

	\end{mdframed}
	
\begin{proof}

If $\alpha = 0$, the result follows directly from Proposition \ref{wP2}. We therefore focus on $0 < \alpha < 1$. Moreover, we consider the case where both 
the supremum and infimum of $\Theta$ are attained on the boundary 
for large time:
\begin{align}
{\rm ess} \sup_{[0,\tau) \times \Omega} \Theta \equiv 	 
\Ov{\Theta}(\tau) &= \max \vtB - \alpha \avintO{ \Theta(\Ov{t}, \cdot) }
\ \mbox{for some}\ 0 \leq \Ov{t} \leq \tau , \br
{\rm ess} \inf_{[0,\tau) \times \Omega} \Theta \equiv\underline{\Theta}(\tau) &= \min \vtB - \alpha \avintO{ \Theta(\underline{t}, \cdot) }
\ \mbox{for some}\ 0 \leq \underline{t} \leq \tau 	 \label{U2}
\end{align}	
for all $\tau > T > 0$. Indeed  
if 
\[
\Ov{\Theta}(\tau) = {\rm ess} \sup_\Omega \Theta(0, \cdot) 
\ \mbox{for all}\ \tau ,  
\]
then, by the minimum principle, 
\begin{align}
\underline{\Theta}(\tau) &\geq \min \left\{ {\rm ess} \inf_\Omega \Theta (0, \cdot);\ \min_{0 \leq t \leq \tau} \{\min_{\partial\Omega} \vtB - \alpha \avintO{ \Theta(t, \cdot) } \} \right\} \br 
&\geq \min \left\{ {\rm ess} \inf_\Omega \Theta (0, \cdot);\ 
 \min_{\partial\Omega} \vtB - \alpha \Ov{\Theta}(\tau) \right\} 
 \ge \min \left\{ {\rm ess} \inf_\Omega \Theta (0, \cdot);\ 
 \min_{\partial\Omega} \vtB - {\rm ess} \sup_\Omega \Theta(0, \cdot) \right\}.
\nonumber
\end{align}
A similar argument can be used to handle the case 
\[	
\underline{\Theta}(\tau) = {\rm ess} \inf_\Omega \Theta(0, \cdot) 
\ \mbox{for all}\ \tau .  
\]	

Finally, we consider the situation specified in \eqref{U2}. 
We have 
\[
\Ov{\Theta}(\tau) \leq \max_{\partial\Omega} \vtB - \alpha \avintO{ 
\Theta(\Ov{t}, \cdot) } \leq \max_{\partial\Omega} \vtB - \alpha \underline{\Theta}(\tau),
\]
and, similarly, 
\[
\underline{\Theta}(\tau) \geq \min_{\partial\Omega} \vtB - \alpha \Ov{\Theta} (\tau). 
\]
Consequently, 
\[
\Ov{\Theta}(\tau) \leq \max_{\partial\Omega} \vtB - \alpha \min_{\partial\Omega} \vtB + \alpha^2 \Ov{\Theta}(\tau) 
\]
yielding the desired conclusion as $0 < \alpha < 1$. The quantity 
$\underline{\Theta}(\tau)$ can be estimated in the same way.

	\end{proof}
	
\section{Levinson dissipativity}
\label{B}

Our next goal is to show that the weak solutions $(\vu, \Theta)$ generate a dissipative system in the sense of Levinson, specifically, any global trajectory will approach a unique bounded set as time goes to infinity. 

\begin{mdframed}[style=MyFrame]

\begin{Theorem}[{\bf Levinson dissipativity}] \label{TB1}
	
	Suppose 
	\[
	0 <\alpha < 1. 
	\] 
	Let 
	$(\vun, \Theta_n)_{n=1}^\infty$ be a sequence of solutions to the OB system 
	defined on the time intervals $(T_n, \infty)$, $T_n \to - \infty$ such that 
	\begin{equation} \label{B2b}
		\| \vun(T_n, ,\cdot) \|_{L^2(\Omega; R^d)} +  \| \Theta_n(T_n, \cdot) \|_{L^\infty(\Omega)} \leq \mathcal{F}_0\ \mbox{uniformly for}\ n \to \infty.	
	\end{equation}	
	
	Then, up to a suitable subsequence, 
	\begin{align}
		\vun &\to \vu \ \mbox{in}\ C_{\rm weak, loc}(R; L^2(\Omega; R^d)) \ \mbox{and weakly in}\ L^2_{\rm loc}(R; W^{1,2}_0 (\Omega; R^d)), \br 
		\Theta_n &\to \Theta \ \mbox{in}\ C_{\rm loc}(R; L^2(\Omega)),
		\ \mbox{weakly-(*) in}\ L^\infty_{\rm loc}(R \times \Ov{\Omega}), 
		\ \mbox{and weakly in} \ L^2_{\rm loc}(R; W^{1,2} (\Omega)), 
		\label{B3}
	\end{align}
	where $(\vu, \Theta)$ is a solution of the OB system defined in $R \times \Omega$ and satisfying
	\begin{equation} \label{B4}
		\| \vu(t, \cdot) \|_{L^2(\Omega; R^3)} + \| \Theta (t, \cdot) \|_{L^\infty(\Omega)} \leq \mathcal{E}_\infty,\ 
		\mathcal{E}_\infty = \mathcal{F}_\infty (\data),\ t \in R.
	\end{equation}
\end{Theorem}

\end{mdframed}

\begin{Remark} \label{RT1}
	
We point out that the constant $\mathcal{F}_\infty$ depends only on the data, meaning specifically the coefficient $\alpha$, $\max_{\partial \Omega} |\vtB|$,
the diffusion coefficients $\mu$, $\kappa$, and the constant $C_p$ in Poincar\' e inequality. In particular, $\mathcal{F}_\infty$ is \emph{independent} of the 
``initial energy'' $\mathcal{F}_0$.	

\end{Remark}

\begin{proof}

{\bf Step 1}
	
To begin, Theorem \ref{UT1} yields a universal bound 
\begin{equation} \label{UU2}
|\Theta_n(t,x) | \leq c(\| \Theta_0 \|_{L^\infty(\Omega)}, \data )
\end{equation}
uniformly for $n \to \infty$.

\medskip

{\bf Step 2}

Next, the energy inequality \eqref{w7} reads 
\[
\frac{\D }{\dt} \intO{ \frac{1}{2} |\vun|^2 } + \mu \intO{ |\Grad \vun |^2 }  
 \leq - \intO{ \Theta_n \Grad G \cdot \vun } 
\leq \| \Theta_n \Grad G \|_{L^\infty([0, \infty) \times \Omega; R^3)} \| \vu_n \|_{L^1(\Omega)}.
\]
As $\vu_n|_{\partial \Omega} = 0$, recalling also \eqref{d3}, 
we may use the uniform bound 
\eqref{UU2} to obtain 
\[
	\limsup_{t \to \infty} \| \vu_n(t, \cdot) \|_{L^2(\Omega; R^d)} \leq 
	c(\| \Theta_0 \|_{L^\infty}, \data) 
\]
uniformly for $n \to \infty$. More specifically, for any $T \in R$
\begin{equation} \label{B1}
\limsup_{n \to \infty} \sup_{t > T}	\| \vu_n(t, \cdot) \|_{L^2(\Omega; R^d)} \leq 	c(\| \Theta_0 \|_{L^\infty}, \data).
\end{equation}

\medskip

{\bf Step 3}

Combining the uniform bounds \eqref{UU2}, \eqref{B1}, with the energy inequalities \eqref{w7}, \eqref{w15} we easily obtain the convergence to a limit 
$(\Theta, \vu)$ claimed in \eqref{B3}. Moreover, the nowadays standard compactness arguments can be used to conclude that the limit $(\vu, \Theta)$ is a weak solution of the OB system in the sense specified in Section \ref{w} defined 
on the entire time interval $t \in R$. 

\medskip 

{\bf Step 4}

It remains to show the universal bound claimed in \eqref{B4}. Let us denote
\[
- \infty < \underline{\Theta} \equiv \inf_{R \times \Omega} \Theta \leq 
\sup_{R \times \Omega} \Theta  \equiv \Ov{\Theta}
\]
and 
\begin{align}
\Ov{\Theta}_{\partial \Omega} &\equiv \sup_{R \times \partial \Omega} \Theta = 
\sup_{t \in R} \left\{ \max_{\partial\Omega} \vtB - \alpha \avintO{\Theta(t, \cdot) } \right\} 
\leq \max_{\partial\Omega} \vtB - \alpha \underline{\Theta},  
\br
\underline{\Theta}_{\partial \Omega} &\equiv \inf_{R \times \partial \Omega} \Theta = 
\inf_{t \in R} \left\{ \min_{\partial\Omega} \vtB - \alpha \avintO{\Theta(t, \cdot) } \right\}
\geq \min_{\partial\Omega} \vtB - \alpha \Ov{\Theta}.
\nonumber
\end{align}

Now, since $\Theta$ is defined for all $t \in R$, we may apply the inequalities 
\eqref{w22} to conclude 
\[
\Theta(t,x) \leq \Ov{\Theta}_{\partial \Omega} \leq \max_{\partial\Omega} \vtB - \alpha \underline{\Theta} \ \mbox{for a.a.}\ (t,x), 
\]
and 
\[
\Theta(t,x) \geq \underline{\Theta}_{\partial \Omega} \geq \min_{\partial\Omega} \vtB - \alpha \Ov{\Theta} \ \mbox{for a.a.}\ (t,x). 
\]
Similarly to the proof of Theorem \ref{UT1} we may infer that 
\[
\Ov{\Theta} \leq \max_{\partial\Omega} \vtB - \alpha \underline{\Theta} \leq 
\max_{\partial\Omega} \vtB + |\min_{\partial\Omega} \vtB| + \alpha^2 \Ov{\Theta}, 
\]
which yields the desired uniform upper bound on $\Ov{\Theta}$. The lower
bound on $\underline{\Theta}$ is obtained in the same way. 

$\Theta$ being uniformly bounded in terms of the data, the same holds true for the $L^2-$norm of the velocity. 

\end{proof}

Motivated by the idea of Sell \cite{SEL} and Ne\v cas, M\' alek \cite{MANE} 
(see also M\' alek and Pra\v z\' ak \cite{MAPR}), we introduce the trajectory 
attractor 
\begin{equation} \label{attra}
\mathcal{A} = \left\{ (\vu, \Theta) \Big| \ \mbox{solution of the OB system}\ 
\in R \times \Omega, \sup_{t \in R} \| \vu(t, \cdot) \|_{L^2(\Omega; R^3)} + \| \Theta (t, \cdot) \|_{L^\infty(\Omega)} \leq \mathcal{F}_\infty \right\}.
\end{equation}

The following properties of $\mathcal{A}$ follow directly from Theorem \ref{TB1}:
\begin{itemize}
	\item 
	$\mathcal{A}$ is non--empty; 
	\item
	$\mathcal{A}$ is time--shift invariant, meaning, 
	\[
	(\vu, \Theta) \in \omega[\vu_0, \Theta_0] \ \Rightarrow\ 
	(\vu, \Theta)(\cdot + T) \in \omega[\vu_0, \Theta_0]
	\]
	for any $T \in R$;
	\item 
	$\mathcal{A}$ is a compact metric space with respect to the topology 
	\[
	(\vu, \Theta) \in C_{\rm loc}(R; L^2(\Omega; R^d)-{\rm weak} \times L^2(\Omega)).
	\]
\end{itemize}
Note carefully that the topology $C_{\rm weak, loc}(R; L^2(\Omega; R^d))$ is metrizable as the velocities in $\mathcal{A}$ belong to a bounded set in $L^2(\Omega; R^d)$.

As a corollary of Theorem \ref{TB1}, we deduce the existence of a uniform bounded 
absorbing set in the $L^2-$norm. 

\begin{mdframed}[style=MyFrame]

\begin{Corollary}[{\bf Bounded absorbing set}] \label{CC1}
	
	Under the hypotheses of Theorem \ref{TB1}, there exists a universal constant 
	$\mathcal{E}_\infty = \mathcal{E}_{\infty}(\rm data)$ such that 
	\[
	\limsup_{t \to \infty} \left(  \| \vu(t, \cdot) \|_{L^2(\Omega; R^d)} 
	+ \| \Theta (t, \cdot) \|_{L^2(\Omega)} \right) \leq \mathcal{E}_\infty 
	\]
	for any weak solution of the OB system defined on $[T, \infty)$ starting from the initial data 
	\[
	\vu(T, \cdot) \in L^2(\Omega; R^d),\ \Theta(T, \cdot) \in L^\infty(\Omega) \cap W^{2,q}(\Omega), 
	\]
	for suitable $q>1$. 
	\end{Corollary}
	
\end{mdframed}

\begin{proof}
	
	First, fix a constant $A$ in such a way that 
\[
\sup_{t \in R} \left( \| \tvu (t, \cdot) \|_{L^2(\Omega; R^2)} + 
\| \widetilde{{\Theta}} \|_{L^2(\Omega)} \right) \leq \frac{A}{2}
\]	
for any $(\tvu, \widetilde{\Theta})$ belonging to the attractor $\mathcal{A}$. It follows from 
Theorem \ref{TB1} that any sequence of times $T_n \to \infty$ contains a subsequence such that 
\[ 
\int_{T_n}^{T_n + 1} \left( \| \vu(t, \cdot) - \tvu_n(t, \cdot) \|_{L^2(\Omega; R^d)} 
+ \| \Theta(t, \cdot) - \widetilde{\Theta}_n(t, \cdot) \|_{L^2(\Omega)} \right) \dt \to 0 
\]
as $T_n \to \infty$ for some $(\tvu_n, \widetilde{\Theta}_n) \in \mathcal{A}$.
We deduce there exists $T > 0$ such that any time interval $[N, N + 1]$, 
$N \geq T$ contains a point $\tau_n$ such that 
\[
\left(  \| \vu(\tau_n, \cdot) \|_{L^2(\Omega; R^d)} 
+ \| \Theta (\tau_n, \cdot) \|_{L^2(\Omega)} \right) \leq A. 
\] 

Finally, the energy inequalities \eqref{w7}, \eqref{w15} imply that 
\[
t \mapsto \left(  \| \vu(t, \cdot) \|_{L^2(\Omega; R^d)} 
+ \| \Theta (t, \cdot) \|_{L^2(\Omega)} \right) 
\]
grows at most exponentially in $t$; whence the desired result follows.
	
\end{proof}

\section{Stationary statistical solutions}
\label{B2}

Given the initial data $(\vu_0, \Theta_0)$, we may define the associated 
$\omega-$limit set, 
\begin{align} 
\omega[\vu_0, \Theta_0] = &\left\{ (\vu, \Theta) \ \Big| 
(\vu, \Theta) \ \mbox{a solution of the OB system in}\ R \times \Omega\right. \br &\left.  
\mbox{obtained as a limit \eqref{B3} with} 
\ \vun(T_n, \cdot) = \vu_0,\ \Theta(T_n, \cdot) = \Theta_0     \right\}.
\label{B5}	
\end{align}	

Given $(\vu_0, \Theta_0)$, Theorem \ref{TB1} yields the following properties 
of the set $\omega[\vu_0, \Theta_0]$:
\begin{itemize}
	\item 
	$\omega[\vu_0, \Theta_0]$ is non--empty; 
	\item
	$\omega[\vu_0, \Theta_0]$ is time--shift invariant, meaning, 
	\[
	(\vu, \Theta) \in \omega[\vu_0, \Theta_0] \ \Rightarrow\ 
	(\vu, \Theta)(\cdot + T) \in \omega[\vu_0, \Theta_0]
	\]
	for any $T \in R$;
	\item 
	$\omega[\vu_0, \Theta_0] \subset \mathcal{A}$ is a compact metric space with respect to the topology 
	\[
(\vu, \Theta) \in C_{\rm loc}(R; L^2(\Omega; R^d)-{\rm weak} \times L^2(\Omega))	.
	\]
	\end{itemize}
	
Now, exactly as in \cite{FanFeiHof}, we may apply the standard
Krylov--Bogolyubov theory to obtain the existence of a stationary statistical solution, meaning, 
a Borel probability measure $\mathcal{M}[\vu_0, \Theta_0]$ supported 
on 	$\omega[\vu_0, \Theta_0]$ such that 
\begin{equation} \label{B6}
	\mathcal{M}[\vu_0, \Theta_0][\mathfrak{B}] = 
	\mathcal{M}[\vu_0, \Theta_0][\mathfrak{B}(\cdot + T)]
\end{equation}
for any Borel set $\mathfrak{B} \subset \omega[\vu_0, \Theta_0]$.

Finally, by the Birkhoff--Khinchin theorem, we obtain the convergence of ergodic averages 
\begin{equation} \label{B7}
  \frac{1}{T} \int_0^T F(\vu, \Theta) (t, \cdot) \ \dt\  \to \Ov{F} \
 \ \mathcal{M}[\vu_0, \Theta_0] - \mbox{a.s.} 
\end{equation}	
whenever $F: W^{-1,2}_0(\Omega) \times L^2(\Omega) \to R$ is bounded, 
Borel measurable, and satisfies 
\[
{\rm exp}_{\mathcal{M}} \left[ F (\vu, \Theta)(0, \cdot) \right] < \infty. 
\]
where ${\rm exp}_{\mathcal{M}}$ denote the expected value.

Summarizing the previous discussion, we obtain:

\begin{mdframed}[style=MyFrame]
	
	\begin{Theorem}[{\bf Stationary statistical solution}] \label{TS1}
		Under the hypotheses of Theorem \ref{UT1}, the OB system admits a
		stationary statistical solution $\mathcal{M}[\vu_0, \Theta_0]$ supported on any $\omega-$limit set 
		$\omega[\vu_0, \Theta_0]$. Moreover, the ergodic averages 
		\[
 \frac{1}{T} \int_0^T F(\vu, \Theta) (t, \cdot) \ \dt\  \to \Ov{F} \ \mbox{as}\ T \to \infty \ 
\ \mathcal{M}[\vu_0, \Theta_0] - \mbox{a.s.} 
\]	
whenever $F: W^{-1,2}_0(\Omega) \times L^2(\Omega) \to R$ is bounded, Borel measurable, and 
satisfies
\[
{\rm exp}_{\mathcal{M}} \left[ F (\vu, \Theta)(0, \cdot) \right] < \infty. 
\]

		\end{Theorem}

	\end{mdframed}

\section{Convergence to equilibria}
\label{C}	

Equilibrium solutions $(\vu_s, \Theta_s)$ are stationary solutions of the (OB) system, 
\begin{align}
	\Div \vu_s &= 0, \br 
	\Div (\vu_s \otimes \vu_s) + \Grad \Pi &= \mu \Del \vu_s - \Theta_s \Grad G
	\br 
	\vu_s \cdot \Grad \Theta_s &= \kappa \Del \Theta_s,	
	\label{es}
\end{align}
with the boundary conditions 
\begin{equation} \label{bes}
\vu_s|_{\partial \Omega} = 0,\ 
\Theta_s|_{\partial \Omega} = \vtB - \alpha \avintO{ \Theta_s }.
\end{equation}
Equivalently, we may rewrite the system in terms of 
\[
\vu_s, \mathcal{T}_s = \Theta_s + \alpha \avintO{ \Theta_s } - \vtB;
\] 
 namely, we have 
\begin{align} 
	\Div \vu_s &= 0, \br
	\Div (\vu_s \otimes \vu_s) + \Grad \Pi &= \mu \Del \vu_s - \mathcal{T}_s \Grad G
	- \vtB \Grad G, \br 
	\vu_s \cdot \big( \Grad \mathcal{T}_s + \Grad \vtB \big)
	&= \kappa \Del \mathcal{T}_s,\ \mathcal{T}_s|_{\partial \Omega} = 0.
	\label{C4}
\end{align}

\subsection{Bounds on equilibrium solutions}
\label{BSS}

Obviously, any equilibrium solution belongs to the global attractor $\mathcal{A}$. In particular, 
\begin{equation} \label{BB1}
\| \vu_s \|_{W^{1,2}_0(\Omega; R^d)} + \| \Theta_s \|_{L^\infty \cap W^{1,2}(\Omega)} \aleq \mathcal{E}_\infty, 	
	\end{equation}
where $\mathcal{E}_\infty$ is the  universal constant depending only on the data, see Remark \ref{RT1}. 
Consequently, one can bootstrap the standard elliptic estimates for the Laplace  
and Stokes operators to conclude that any stationary solution is strong 
as far as we additionally assume, say
\begin{equation} \label{BB2}
\Grad G , \Grad \vtB \in L^\infty(\Omega; R^d) \, \text{ which imply }\,\, 
\vu_s \in W^{2,q}(\Omega; R^d),\ \Theta_s \in W^{2.q}(\Omega),\ 1 \leq q < \infty,  
\end{equation}
at least if $\partial \Omega$ is of class $C^2$ as required by the elliptic $L^p-$theory.

\subsection{Relative energy inequality}

As the next step, we derive the relative energy inequalities for the velocity and temperature.
First, we rewrite the energy inequality \eqref{w7} in terms of $\mathcal{T}$:
\begin{equation} \label{C1}
	\frac{\D }{\dt} \intO{ \frac{1}{2} |\vu|^2 } + 
	\mu \intO{ |\Grad \vu |^2 } \leq - \intO{ \mathcal{T} \Grad G \cdot \vu }
	- \intO{ \vtB \Grad G \cdot \vu }.
	\end{equation}

Now, for $\tvu$ sufficiently smooth and satifying 
\[
\Div \tvu = 0,\ \tvu|_{\partial \Omega} = 0,
\]
we deduce 
\begin{align} 
\frac{\D }{\dt} \intO{ \frac{1}{2} |\vu - \tvu |^2 } &+ \mu \intO{ \Grad \vu \cdot (\Grad \vu - \Grad \tvu)  }\br &\leq - \intO{ \mathcal{T} \Grad G \cdot (\vu - \tvu) } - \intO{ \vtB \Grad G \cdot (\vu - \tvu) } \br & 
+ \intO{ \big( (\tvu - \vu) \cdot \partial_t \tvu - (\vu \otimes \vu) : \Grad \tvu \big) }.
\label{C2}
	\end{align}

Similarly, we may rewrite the ``thermal'' energy balance as
\begin{align}
	\frac{\D}{\dt} &\intO{  \frac{1}{2} \Lambda \big(\mathcal{T} - \wmT \big)
	\big( \mathcal{T} - \wmT \big)  
	  } + 
	\kappa \intO{ \Grad \mathcal{T} \cdot \big(\Grad \mathcal{T} - \Grad \wmT \big)  } \br 
	&= \intO{ \vtB \vu \cdot \big( \Grad \mathcal{T} - \Grad \wmT \big)} 
	+ \intO{ \big( \Lambda (\wmT ) - \Lambda (\mathcal{T}) \big) \partial_t \wmT } \br
&	 
+	\intO{ \mathcal{T} \vu \cdot \big( \Grad \mathcal{T} - \Grad \wmT \big) }
\label{C3}
\end{align}
for any sufficiently smooth $\wmT$, $\wmT|_{\partial \Omega} = 0$.

\subsection{Stability of equilibrium solutions}

Consider the special case of the relative energy inequality, where 
$(\tvu, \wmT) = (\vu_s, \mathcal{T}_s)$ is an equilibrium solution satisfying \eqref{C4}.

A straightforward manipulation yields
\begin{align}
	\frac{\D }{\dt} \intO{ \frac{1}{2} |\vu - \vu_s |^2 }
	& + \mu \intO{ \Grad \vu \cdot \big(\Grad \vu - \Grad \vu_s\big)  }\br
	&\leq - \intO{ \mathcal{T} \Grad G \cdot (\vu - \vu_s) } - \intO{ \vtB \Grad G \cdot (\vu - \vu_s) } \br & 
	- \intO{  (\vu \otimes \vu) : \Grad \vu_s  }, 
\nonumber	
	\end{align}
where 
\[
 \intO{ (\vu_s \otimes \vu_s) : \Grad \vu } +
\mu \intO{ \Grad \vu_s \cdot \Grad (\vu_s - \vu) } 
= \intO{ \big( \mathcal{T}_s + \vtB \big) \Grad G \cdot (\vu - \vu_s) }.
\]	
Consequently, 
\begin{align}
	\frac{\D }{\dt} \intO{ \frac{1}{2} |\vu - \vu_s |^2 } &
	+ \mu \intO{ \big| \Grad \vu - \Grad \vu_s \big|^2 }\br
	&\leq - \intO{ \big( \mathcal{T} - 
		\mathcal{T}_s \big) \Grad G \cdot (\vu - \vu_s) }  \br & 
	- \intO{  (\vu \otimes \vu) : \Grad \vu_s  }  -
	\intO{ (\vu_s \otimes \vu_s) : \Grad \vu }, 
	\nonumber
\end{align}
where 
\begin{align*}
  \intO{  \Big[  (\vu \otimes \vu) : \Grad \vu_s + (\vu_s \otimes \vu_s) : \Grad \vu \Big] } 
   & = \intO{ \Grad \vu_s : \big( \vu \otimes (\vu - \vu_s)\big) } \\
   & = \intO{ \Grad \vu_s : \big((\vu - \vu_s ) \otimes (\vu - \vu_s)\big) }. 
\end{align*}
Thus we conclude
\begin{align}
	& \frac{\D }{\dt} \intO{ \frac{1}{2} |\vu - \vu_s |^2 } + \mu \intO{ \big| \Grad \vu - \Grad \vu_s \big|^2 }\br 
	& \qquad\quad \leq - \intO{ \big( \mathcal{T} - 
		\mathcal{T}_s \big) \Grad G \cdot (\vu - \vu_s) } 
		- \intO{ \Grad \vu_s : \big((\vu - \vu_s ) \otimes (\vu - \vu_s)\big) } .
	\label{C5}
\end{align}

Similarly, using \eqref{C3} we get 
\begin{align}
	\frac{\D}{\dt} &\intO{  \frac{1}{2} \Lambda \big(\mathcal{T} - \mathcal{T}_s \big) 
	     \big( \mathcal{T} - \mathcal{T}_s \big) } 
	+ \kappa \intO{ \Grad \mathcal{T} \cdot \big(\Grad \mathcal{T} - \Grad \mathcal{T}_s \big)  } \br 
	&= \intO{ \vtB \vu \cdot \big( \Grad \mathcal{T} - \Grad \mathcal{T}_s \big)}  \br
	&	+	\intO{ \mathcal{T} \vu \cdot \big( \Grad \mathcal{T} - \Grad \mathcal{T}_s \big) },
	\nonumber
\end{align}
and 
\[
\kappa \intO{ \Grad \mathcal{T}_s \cdot \Grad \big( \mathcal{T}  - \mathcal{T}_s \big)} 
= \intO{ \vu_s \cdot \big(\Grad \mathcal{T}_s + \Grad \vtB \big) \big( \mathcal{T}_s - \mathcal{T} \big)}.  
\]
Consequently, 
\begin{align}
	\frac{\D}{\dt} &\intO{  \frac{1}{2} \Lambda \big(\mathcal{T} - \mathcal{T}_s \big) 
	\big( \mathcal{T} - \mathcal{T}_s \big)  
	} + 
	\kappa \intO{ \big| \Grad \mathcal{T} - \Grad \mathcal{T}_s \big|^2  } \br 
	&= \intO{ \vtB (\vu - \vu_s) \cdot \big( \Grad \mathcal{T} - \Grad \mathcal{T}_s \big)}  \br
	&	 
	+	\intO{ \mathcal{T}_s (\vu - \vu_s) \cdot \big( \Grad \mathcal{T} - \Grad \mathcal{T}_s \big) }.
	\nonumber
\end{align}
Thus we may infer that 
\begin{align}
	\frac{\D}{\dt} &\intO{  \frac{1}{2} \Lambda \big(\mathcal{T} - \mathcal{T}_s \big)
	   \big( \mathcal{T} - \mathcal{T}_s \big) } 
	   + \kappa \intO{ \big| \Grad \mathcal{T} - \Grad \mathcal{T}_s \big|^2  } \br &=
	 \intO{  \big(  \mathcal{T} - \mathcal{T}_s \big) \Grad \big( \vtB + \mathcal{T}_s \big) \cdot (\vu_s - \vu) }  .
	\label{C6}
\end{align}

Summing up \eqref{C5}, \eqref{C6} we conclude 
\begin{align} 
		\frac{1}{2}\frac{\D }{\dt} &\intO{ \left(  |\vu - \vu_s |^2 
		+  \Lambda \big(\mathcal{T} - \mathcal{T}_s \big) \big( \mathcal{T} - \mathcal{T}_s \big) \right) } 
		+ \intO{ \left( \mu \big| \Grad \vu - \Grad \vu_s \big|^2 
		  + \kappa \big| \Grad \mathcal{T} - \Grad \mathcal{T}_s \big|^2\right)}\br 
		  &\leq - \intO{ \big( \mathcal{T} - \mathcal{T}_s \big) \Grad G \cdot (\vu - \vu_s) } - \intO{ 
		\Grad \vu_s : \big( (\vu - \vu_s) \otimes (\vu - \vu_s) \big)} \br
		&\quad +  \intO{ \big(  \mathcal{T} - \mathcal{T}_s \big)
		     \Grad \big( \vtB + \mathcal{T}_s \big) \cdot (\vu_s - \vu) }  .
	\label{C7}
	\end{align}

It follows from \eqref{C7} that any global weak solution converges to an equilibrium as long as the quantities 
\[
\| \Grad G \|_{L^\infty(\Omega; R^d)}, 
\ \| \Grad \vu_s \|_{L^\infty(\Omega; R^{d \times d})},\ 
\| \Grad (\vtB + \mathcal{T}_s ) \| _{L^\infty(\Omega; R^d)} 
\]
are sufficiently small. As observed in Section \ref{BSS}, the stationary solutions are small as long as the the data $\Grad G$, $\Grad \vtB$ are small. 
Consequently, we have proved the following result.

\begin{mdframed}[style=MyFrame]

\begin{Theorem}[{\bf Stability of equilibria}] \label{TC1}
	
	Under the hypotheses of Theorem \ref{TB1}, there exists $\delta > 0$ depending solely on the
	Poincar\' e constant of the domain $\Omega$ and 
	the diffusion coefficients $\mu$, $\kappa$, and 
	$\max_{\partial \Omega}|\vtB|$ such that for any data $G$, $\vtB$, 
	\[
	\| \Grad \vtB \|_{L^\infty(\Omega; R^d)} + \| \Grad G \|_{L^\infty(\Omega; R^d)} < \delta, 
	\]
	the OB system admits a unique stationary solution $(\vu_s, \Theta_s)$ that is globally exponentially stable in the $L^2(\Omega)-$norm, meaning 
	\[ 
	\| \vu(t, \cdot) - \vu_s \|_{L^2(\Omega; R^d)} + 
	\| \Theta(t, \cdot) - \Theta_s \|_{L^2(\Omega)} \leq C 	\Big[ \| \vu(0, \cdot) - \vu_s \|_{L^2(\Omega; R^d)} + 
	\| \Theta(0, \cdot) - \Theta_s \|_{L^2(\Omega)} \Big] \exp (-K t) 
\]
for some absolute constants $C$ and $K > 0$ and any solution $(\vu, \Theta)$ of the OB system.
	
	\end{Theorem}

\end{mdframed}

\subsection{Convergence for the classical B\' enard problem with 
temperature gradient aligned to gravitation}

In the standard setting of the Rayleigh--B\' enard problem, see e.g. \cite{BoPeAh}, 
the temperature gradient is parallel to the gravitational force, 	 
\[
\Grad \vtB \times \Grad G = 0.
\]
Then $\vu_s = 0$, and $\Theta_s = \vtB$ (harmonic extension of the boundary data) 
is the corresponding stationary solution. 
Revisiting formulae \eqref{C5}, \eqref{C6} with $\vu_s = 0$ we observe
that the stability result stated in Theorem \ref{TC1} holds if 
either $\Grad \vtB$ or $\Grad G$ is small in the $L^\infty-$norm. 
Indeed, for
\[
A = \| \vu  \|_{L^2(\Omega; R^d)},\ 
B = \| \mathcal{T} - \mathcal{T}_s \|_{L^2(\Omega)}, 
\] 
noting as $C_p$ the constant in the Poincar\'e inequality written in the form
\[
  C_p \| v \|_{L^2(\Omega)} \le \| \nabla v \|_{L^2(\Omega;R^d)} 
\]
for $v\in W^{1,2}_0(\Omega)$, 
the problem reduces to finding two positive constants $a,b$ such that the quadratic form 
\[
a \mu C_p^2 A^2 - a |\Grad G|_{L^\infty} AB + 
b \kappa C_p^2 B^2 - b |\Grad \vtB|_{L^\infty} AB,  
\]
is strictly positively definite, meaning 
\[
 \left( Z |\Grad G|_{L^\infty} + \frac{1}{Z} |\Grad \vtB|_{L^\infty}  \right) 
  < 2 C_p^2 \sqrt{\mu \kappa} ,\ Z = \left( \frac{a}{b} \right)^{\frac{1}{2}}. 
\]
We have thus shown the following result for the Rayleigh--B\' enard problem with temperature gradient aligned to gravitation. 

\begin{mdframed}[style=MyFrame]
	
	\begin{Theorem}[{\bf Stability of equilibria - aligned temperature gradient}] \label{TBB1}
		
In addition to the hypotheses of Theorem \ref{TC1} suppose 
\[
\Grad \vtB \times \Grad G = 0, 
\]
and 
\begin{equation} \label{RB}
\| \Grad G \|_{L^\infty(\Omega; R^d)} \| \Grad \vtB \|_{L^\infty(\Omega; R^d)}
\leq C_p^2 \mu \kappa,
\end{equation}
where $C_p$ is the constant in Poincar\' e inequality.

Then there are absolute constants $C$, $K > 0$ such that 
\[ 
\| \vu(t, \cdot)  \|_{L^2(\Omega; R^d)} + 
\| \Theta(t, \cdot) - \Theta_s \|_{L^2(\Omega)} \leq C 	\Big[ \| \vu(0, \cdot)   \|_{L^2(\Omega; R^d)} + 
\| \Theta(0, \cdot) - \Theta_s \|_{L^2(\Omega)} \Big] \exp (-K t) 
\]
for any weak solution $(\vu, \Theta)$ of the OB system.

	\end{Theorem}
	
\end{mdframed}

Obviously, condition \ref{RB} corresponds to smallness of the Rayleigh number necessary for equilibrium stability.

\def\cprime{$'$} \def\ocirc#1{\ifmmode\setbox0=\hbox{$#1$}\dimen0=\ht0
	\advance\dimen0 by1pt\rlap{\hbox to\wd0{\hss\raise\dimen0
			\hbox{\hskip.2em$\scriptscriptstyle\circ$}\hss}}#1\else {\accent"17 #1}\fi}

\end{document}